\documentclass{TPmod}
\usepackage{url}
\usepackage{setspace}
\usepackage{scrextend}

\usepackage{amsthm}
\usepackage{amssymb}
\usepackage[capitalize]{cleveref}
\usepackage{mathtools}
\usepackage{xy}
\usepackage{tikz}
\usepackage{tikz-cd}
\usetikzlibrary{shapes.geometric,arrows}

\DeclareMathOperator{\Aut}{Aut}
\DeclareMathOperator{\Auteq}{Aut}

\newcommand{\scrM}{\EuScript{M}}
\newcommand{\scrA}{\EuScript{A}}
\newcommand{\scrX}{\EuScript{X}}
\newcommand{\scrJ}{\EuScript{J}}
\newcommand{\scrF}{\EuScript{F}}

\DeclareMathOperator{\Diff}{Diff}

\DeclareMathOperator{\Symp}{Symp}

\def\bP{\mathbb{P}}

\def\bQ{\mathbb{Q}}

\def\bZ{\mathbb{Z}}

\renewcommand{\cong}{\simeq}

\begin{document}
\thispagestyle{empty}
\title{Irrationality and monodromy for cubic threefolds}

\author[Ivan Smith]{Ivan Smith}
\thanks{The author is partially supported by a Fellowship from EPSRC.}

\address{Ivan Smith, Centre for Mathematical Sciences, University of Cambridge, Wilberforce Road, Cambridge CB3 0WB, U.K.}

\begin{abstract} {\sc Abstract:} Let $\scrM$ denote the moduli space of smooth cubic threefolds. We prove that the  monodromy $\pi_1(\scrM) \to Sp(10;\bZ)$ of the 3rd cohomology groups on the fibres of the universal family does not factor through the genus five  mapping class group. This gives a geometric group theory perspective on the well-known irrationality of cubic threefolds.
\end{abstract}
\thispagestyle{empty}
\maketitle
\section{Introduction}

Clemens' and Griffiths' renowned proof \cite{Clemens-Griffiths} of the irrationality of smooth cubic threefolds has two main ingredients.  The first, more elementary, part comprises the observation that if a threefold is rational, then its intermediate Jacobian is isomorphic as a principally polarised abelian variety to a product of Jacobians of algebraic curves.  The second, deeper part, is to prove that the intermediate Jacobian of a cubic is not such a product, by showing that its $\Theta$-divisor has  singularity locus of codimension $4$.  (The $\Theta$-divisor of a Jacobian is singular in codimension $3$ for smooth non-hyperelliptic curves, in codimension $2$ for smooth hyperelliptic curves, and in codimension one when reducible, i.e. for the product of Jacobians of a disjoint union of curves of positive genera.)  This note suggests an alternative perspective on the second part,  based on ideas of monodromy. 

There are many proofs of irrationality of (general) cubic threefolds, via Hodge theory, the Weil conjectures, motivic integration, degeneration of Prym varieties etc \cite{Murre, Collino, Gwena, MR, KT}.  The viewpoint presented here does not yield a new proof but seems philosophically related. It was motivated by considerations in symplectic topology, cf. Section \ref{Sec:aspirational}, and suggests a 
connection, allbeit partial, between rationality of threefolds and rigidity questions in geometric group theory. 

\subsection{Result}

Let $\scrM_{3,3}$ denote the moduli space of smooth cubic threefolds, and $\scrA_5$ the moduli space of five-dimensional principally polarised abelian varieties.  The association $X \mapsto IJ(X)$ of the intermediate Jacobian to a cubic threefold $X \subset \bP^4$ defines a map $IJ: \scrM_{3,3} \to \scrA_5$, and there is an associated map on orbifold fundamental groups
\begin{equation} \label{eqn:monodromy}
\pi_1(\scrM_{3,3}) \to \pi_1(\scrA_5) = Sp(10;\bZ).
\end{equation} 
The map \eqref{eqn:monodromy} can also be understood purely topologically, as the monodromy action on $H^3(X;\bZ)$ (a lattice of rank 10 with its skew-symmetric intersection form) arising from parallel transport for the Gauss-Manin connection in the universal family of cubic threefolds over $\scrM_{3,3}$.

For $g>0$, let $\Gamma_g = \pi_0\Diff^+(\Sigma_g)$ denote the mapping class group of a genus $g$ surface, which comes with a natural map $\Gamma_g \to Sp(2g;\bZ)$.  More generally, if $g_i>0$ are integers with $\sum_i g_i = 5$, and $\Sigma = \sqcup_i \Sigma_{g_i}$ is a closed surface of `total genus' $5$, there is a natural map $\Gamma(\Sigma) = \pi_0\Diff^+(\Sigma) \to Sp(10;\bZ)$. Our main result is:

\begin{thm} \label{thm:main}
The monodromy \eqref{eqn:monodromy} does not factor through the mapping class group.
\end{thm}

We focus on showing it does not factor through $\Gamma_5$; the other cases are simpler, see Corollary \ref{cor:disconnected}.  We remark that there is no rational cohomological obstruction to a factorization \cite{Peters-Steenbrink}.

\subsection{Irrational context\label{Sec:irrat}}

Let $M_g$ denote the moduli space of genus g curves. The Torelli map $\frak{t}: M_g \to \scrA_g$ is injective on geometric points, by the classical Torelli theorem.  Denote the image of $\frak{t}$ by $\scrJ_g$, the locus of Jacobians of genus g curves.  Let $\overline{\scrJ}_g$ denote the closure of $\scrJ_g$ inside $\scrA_g$; both $\scrJ_g$ and $\overline{\scrJ}_g$ are singular. Let $M_g^{ct}$ denote the moduli space of curves of compact type (ones whose dual graph is a tree).  The Torelli map extends to a map $\frak{t}: M_g^{ct} \to \overline{\scrJ}_g \subset \scrA_g$ which, however, is no longer injective (the Jacobian of a curve of compact type depends only on its normalization, not the points at which the component curves are glued). We have a diagram of spaces
\begin{equation}\label{eqn:lift}
\begin{gathered}
\begin{tabular}{c}
\xymatrix{
\scrM_{3,3} \ar[r] \ar@{-->}[rd] & \scrA_5 \ar@{=}[r] & \scrA_5 \\ & \overline{\scrJ}_5 \ar[u] &  \scrJ_5 \ar[l] \ar[u] \\ & M_5^{ct} \ar[u] & M_5 \ar[u] \ar[l] 
}
\end{tabular}
\end{gathered}
\end{equation}
Suppose for a moment that every smooth cubic was rational, so its intermediate Jacobian was a product of Jacobians of curves. This would exactly say that the top left horizontal map $IJ$ lifts to the dotted arrow, i.e. that it lands in $\overline{\scrJ}_5$, and hence the induced map  $\pi_1(\scrM_{3,3}) \to Sp(10;\bZ)$ factors through $\pi_1(\overline{\scrJ}_5)$.   This would certainly hold if the monodromy further lifted to $M_5^{ct}$ (which is a smooth orbifold, so somewhat more accessible than the singular space $\overline{\scrJ}_5$).  The fundamental group of $M_g^{ct}$ is the quotient of the mapping class group by the subgroup generated by Dehn twists in separating curves. This is known through work of Johnson \cite{Johnson} to yield an extension (for $g>2$)
\[
1 \to \Lambda^3(H) / H \to \pi_1(M_g^{ct}) \to Sp(2g;\bZ) \to 1
\]
where $H = H^1(\Sigma_g;\bZ)$ carries its symplectic form $[\omega]$, and wedging with $[\omega]$ defines the inclusion $H \hookrightarrow \Lambda^3H$. This extension splits rationally, since $\Lambda^3(H)/H$ is an algebraic representation; it does not split integrally, but is defined by a class of order two \cite{Morita, Hain-geometric}.  It would be interesting to compute the pullback of this class to $H^2(\pi_1(\scrM_{3,3}), \Lambda^3(H)/H)$.

This paper has a  different goal.  En route to proving irrationality, Clemens and Griffiths show that the intermediate Jacobian of a smooth cubic threefold is irreducible as a principally polarised abelian variety. If one grants this (deep) fact, then rationality would entail that $IJ(X)$ was a genus 5 Jacobian,  rather than a product of Jacobians  of perhaps smaller genus curves; the corresponding lift in \eqref{eqn:lift} would then be to the middle of the right hand column, i.e. to $\scrJ_5 \subset \overline{\scrJ}_5$.  For $g>2$, the map $M_g \to J_g \subset \scrA_g$ is not an immersion along the hyperelliptic locus (it is ramified as a map of stacks, or concretely as a map of varieties if one equips the curves and Jacobians with level $\ell$ structures for odd $\ell \geq 3$, essentially because the automorphism $-1$ of any Jacobian arises from an automorphism of a curve only when the curve is hyperelliptic, cf. \cite{Catanese}). It follows (for instance) from results of Hain \cite{Hain, Hain-torelli} that the orbifold fundamental group of $\scrJ_g$ differs from that of $M_g$, but a concrete description of the former does not seem to appear in the literature. One can again, however, ask if the stronger result is true: does the monodromy of the family of cubic threefolds further lift under the Torelli map, to the bottom line in \eqref{eqn:lift}, i.e. to $M_5$?    It is exactly this which Theorem \ref{thm:main} obstructs.  

A family of irreducible Jacobians arises from a family of curves when it is disjoint from or wholly contained in the locus of hyperelliptic Jacobians \cite{Hain}. Thus, Theorem \ref{thm:main} rules out the (admittedly special) possibility that all smooth cubic threefolds are rational with intermediate Jacobian $IJ(X)$ having theta-divisor with singular locus of fixed codimension $\geq 2$. Indeed, if that codimension were $\geq 4$ we would be done by the opening remarks of the paper; if it were $2$ or $3$ one would obtain a lift of the monodromy to $M_5$, contradicting Theorem \ref{thm:main}.
(If the fixed codimension is $1$, Theorem \ref{thm:main} gives an obstruction  to rationality if one knows \emph{a priori} that $IJ(X)$ has a fixed `reducibility type' into pieces whose theta-divisors in turn have fixed codimension singularities.) 
\begin{rmk}
The Klein cubic $\{\sum_{j\in\bZ/5} x_j^2x_{j+1}=0\} \subset \bP^4$ has automorphism group $\bP SL_2(\mathbb{F}_{11}))$, see \cite{Adler}.  Beauville \cite{Beauville-sextic} observed that, since this group is too large to act on the Jacobian of any union of small genus curves, this threefold cannot  be rational. Since  $\overline{\scrJ}_5 \subset \scrA_5$ is closed, it immediately follows that the general cubic threefold is not rational.  (The deep and general recent specialisation result on rationality due to Kontsevich and Tschinkel \cite{KT} also yields irrationality of very general cubic threefolds from that of the Klein cubic.)  Gwena \cite{Gwena} gave an obstruction to rationality from `local monodromy' considerations for the Prym varieties of a family of cubics degenerating to the Segre cubic (which has the maximal number $10$ of nodes of any cubic threefold),  see also \cite{Hulek}. This, and the work of Hain \cite{Hain}, seem close in spirit to the viewpoint given here. 
\end{rmk}  

\subsection{Aspirational context\label{Sec:aspirational}}

The original motivation for this paper concerned  `symplectic birationality' -- can one obtain a cubic threefold as a symplectic manifold, by blowing up and down projective space in symplectic submanifolds, and allowing deformations of the symplectic form (the latter is natural since symplectic blow-up already depends on a choice of scale, so yields a manifold well-defined up to symplectic deformation rather than symplectomorphism)?  Since the intermediate pieces needn't now be algebraic, or have Hodge structures, this is not something classical algebro-geometric techniques say anything about; and at least in higher dimensions there are plenty of symplectic non-algebraic submanifolds to blow up and down in.  There is a heuristic and incomplete picture of the behaviour of the derived Fukaya category $\scrF(X)^{per}$ of a symplectic manifold $X$ under blowing up \cite{Katzarkov:survey, Smith:HFquadrics, Seidel:flux}, which informally suggests that a `symplectically rational' six-manifold should have Fukaya category over a characteristic zero field which is `closely related' to\footnote{perhaps a sheaf of Clifford algebras over, or a deformation thereof}  the Fukaya category of a possibly disconnected smooth symplectic surface $\Sigma$, along with some benign semisimple summands (arising from $\scrF(\bP^3)^{per}$ and the contribution of blowing up in points). It seems hard to make that  rigorous at the current development of the subject, but it suggests that if cubics were symplectically rational, the symplectic parallel transport map 
\[
\pi_1(\scrM_{3,3}) \to \pi_0\Symp(X) \to \Auteq(\scrF(X)^{per}) \to \Aut\, HH^*(\scrF(X)^{per}) = \Aut\,QH^*(X) 
\]
might factor through a map
\[
\pi_1(\scrM_{3,3}) \to \Auteq(\scrF(\Sigma)^{per})
\]
and the latter is known to surject onto the mapping class group $\Gamma(\Sigma)$, cf. \cite{Auroux-Smith}.  Note that in this context it does appear that the mapping class group $\Gamma_g=\pi_1(M_g)$ is the more relevant object, and not its less accessible cousin $\pi_1(\scrJ_g)$. Thus, we hope that Theorem \ref{thm:main} will be relevant to an eventual theory of symplectic non-rationality. 

\subsection{Actual context}

Theorem \ref{thm:main} fits into general ideas of (super)rigidity for mapping class groups. We make particular use of the near-sharp constraints on homomorphisms of braid groups to mapping class groups established by Castel in \cite{Castel}. There is a presentation of $\pi_1(\scrM_{3,3})$ due to L\"onne \cite{Lonne} which realise it as the quotient of an Artin group $G(\Gamma)$ associated to a Dynkin-type graph $\Gamma$ coming from unfolding the isolated Fermat singularity $\{\sum_{j=0}^4 z_j^3 = 0\}\subset\C^4$, cf. \cite{Lonne-BP}. A lift of  \eqref{eqn:monodromy} to $\Gamma_5$ (say) would yield a homomorphism $G(\Gamma) \to \Gamma_5$. The proof of Theorem \ref{thm:main} has two steps: to show that under this homomorphism, the generators of the Artin group $G(\Gamma)$ corresponding to vertices of $\Gamma$ are taken to Dehn twists in non-separating simple closed curves; and then to rule out the existence of a configuration of curves with the necessary intersection pattern, by the `change-of-coordinates' principle \cite{Farb-Margalit}.  (Compare also to Salter's work \cite{Salter}, which involves realising rather than obstructing interesting Brieskorn-Pham configurations of curves.) This last step amounts to saying that certain graphs don't embed in the `Schmutz graph' \cite{SchmutzSchaller} of a small genus surface, and can be compared to similar recent investigations of the finite subgraphs of complexes of curves \cite{ABG}.  

Ultimately, such rigidity results are applications of Thurston's dynamical classification of mapping classes \cite{FLP}. The methods are effective at constraining homomorphisms to $\Gamma_g$, but probably not to subquotients such as $\pi_1(\scrJ_g)$ even given a detailed picture of that group.

\begin{rmk}    It is also interesting to ask whether the Clemens-Griffiths theorem implies Theorem \ref{thm:main}.  Suppose \eqref{eqn:monodromy} did factor through the mapping class group. Since the relevant spaces are orbifold $K(\pi,1)$'s, this would define a smooth map from $\scrM_{3,3} / \bP GL(5)$ to $M_5 \subset \scrA_5$. One could conceivably use negative curvature arguments in the vein of \cite{DKW} to replace this with a harmonic map, and then rigidity results for harmonic maps \cite{Yue} to force this to be (anti)holomorphic, contradicting \cite{Clemens-Griffiths}.  
\end{rmk}

\begin{rmk} \label{rmk:fano} Classical algebro-geometric constructions yield  non-trivial homomorphisms which might be interesting in the context of rigidity for mapping class groups. 
Let $\scrM_{3,L}$ be the moduli space of pairs comprising a cubic threefold $X$ with a line $\bP^1 = L \subset X$, which maps to $\scrM_{3,3}$ with fibres the `Fano surfaces of lines' on $X$.  (The Fano surface $S$ of lines on $X$ has $H^1(S;\bZ) = \bZ^{10}$, and the fundamental group of $\scrM_{3,L}$ is very different from that of $\scrM_{3,3}$.)  The intermediate Jacobian map $\scrM_{3,L} \to \scrA_5$  factors through the space of Prym curves of genus $6$, an unbranched cover of $M_6$ with fibre $H^1(\Sigma_6;\bZ/2)^{\times}$. Thus,  $\pi_1(\scrM_{3,L}) \to Sp(10;\bZ)$ does factor through an index  $63$ subgroup of $\Gamma_6$.
\end{rmk}

\paragraph{Acknowledgements.} This note develops a suggestion of Simon Donaldson. I am grateful to Klaus Hulek, Oscar Randal-Williams, Dhruv Ranganathan, Paul Seidel and Henry Wilton for helpful conversations and comments; to Richard Hain and Filippo Viviani for pointing out errors in the contextual discussion of Section \ref{Sec:irrat}  in the first version of this paper;  and to the Engineering and Physical Sciences Research Council, U.K., for financial support. 

\section{Background}

\subsection{Fundamental group of the moduli space of cubics}

We summarise some results of L\"onne \cite{Lonne}.  Consider the lexicographic ordering on the sixteen element set $\{0,1\}^4$. We define a graph $\Gamma$ with vertices indexed by this set, and with an edge between ${\bf i}$ and ${\bf j}$ whenever for each $\mu,\nu \in \{1,2,3,4\}$ one has $(i_\mu - j_\mu)(i_\nu - j_\nu) \geq 0$.  Thus, the obstruction to there being an edge between $(i_1,\ldots,i_4)$ and $(j_1,\ldots,j_4)$ is exactly that, for some pair of places $\mu,\nu \subset \{1,2,3,4\}$, the corresponding tuples ${\bf i}$ and ${\bf j}$ take opposite values $(0,1)$ and $(1,0)$. The analogous graph on $\{0,1\}^k$ for $k=3$ is drawn below; we are interested in the four-dimensional hypercube version. \newline

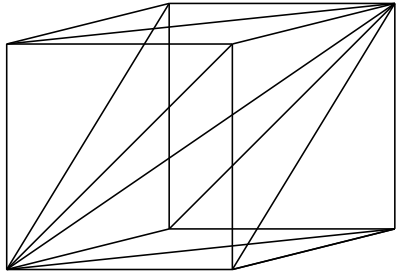
\begin{figure}[ht]
\begin{center} 
\begin{tikzpicture}[scale=0.03]

\draw [semithick] (0,100) -- (100,100);
\draw[semithick] (0,0) -- (0,100);
\draw[semithick] (100,0) -- (100,100);
\draw[semithick] (0,0) -- (100,0);

\draw[semithick] (0,0) -- (72,18);
\draw[semithick] (100,0) -- (172,18);
\draw[semithick] (0,100) -- (72,118);
\draw[semithick] (100,100) -- (172,118);
\draw[semithick] (72,18) -- (72,118);

\draw [semithick] (0,0) -- (100,100); 
\draw[semithick] (100,0) -- (172,18);

\draw[semithick] (0,0) -- (172,18); 
\draw[semithick] (0,100) -- (172,118);
\draw[semithick] (0,0) -- (72,118);
\draw[semithick] (0,0) -- (172,118);

\draw[semithick] (72,18) -- (172,18);
\draw[semithick] (72,118) -- (172,118);
\draw[semithick] (72,18) -- (172,118);
\draw[semithick] (172,18) -- (172,118);
\draw[semithick] (100,0) -- (172,118);

\end{tikzpicture}
\end{center}
\caption{The Artin graph for cubic surfaces\label{Fig:2d_case}}
\end{figure}

Consider the Artin group $G(\Gamma)$ generated by elements $\sigma_v$, for $v \in \mathrm{Vert}(\Gamma)$, and with relations
\begin{equation} \label{eqn:Artin_relations}
[\sigma_v,\sigma_w] = 1 \ \mathrm{if} \  (v,w) \not \in \mathrm{Edge}(\Gamma); \ \sigma_v\sigma_w\sigma_v = \sigma_w\sigma_v\sigma_w \ \mathrm{if} \  (v,w) \in \mathrm{Edge}(\Gamma).
\end{equation}
We will call the generators $\sigma_v$ the `standard generators' of $G(\Gamma)$.

\begin{rmk}\label{rmk:extremal}
The vertices $(0000)$ and $(1111)$ are connected by edges to all others; we will call these two vertices \emph{extremal}, and denote either by $v_{ext}$.  For any non-extremal vertex $v$, there is another vertex $v'$ with $[\sigma_v, \sigma_{v'}] = 1$. 
\end{rmk}

We have a further collection of `triangle relations'
\begin{equation} \label{eqn:triangle_relations}
\sigma_u \sigma_v \sigma_w \sigma_u = \sigma_v \sigma_w \sigma_u \sigma_v \quad  \mathrm{when} \ \ \{(u,v), (v,w), (w,u)\} \subset \mathrm{Edge}(\Gamma).
\end{equation}

Let $\scrM_{3,3} \subset \bP H^0(\mathcal{O}_{\bP^4}(3))$ be the moduli space of smooth cubic threefolds, i.e. the complement of the discriminant divisor $\Delta_{3,3}$ in the linear system of cubics on $\bP^4$.  

\begin{prop}[L\"onne]
The fundamental group $\pi_1(\scrM_{3,3})$ is a quotient of $G(\Gamma)$ by the triangle relations, and a collection of `non-local' relations given in \cite[Section 7]{Lonne}. 
\end{prop}

\begin{rmk} \label{rmk:Milnor} The variety $\{\sum_{i=0}^3 x_i^3 = 0\}$ defines an isolated hypersurface singularity of Milnor number $16$. The graph $G(\Gamma)$ occurs as the intersection graph of a distinguished basis of vanishing cycles for the Milnor fibre $X_{aff} = \{\sum_{i=0}^3 x_i^3 = 1\}$, see \cite{Futaki-Ueda}.  Let $X_{aff} \subset X$ denote the inclusion of the Milnor fibre in its projective closure (i.e. in a smooth projective cubic threefold). Then there is a diagram
\[
\xymatrix{
G(\Gamma) \ar[r] \ar[d] & \pi_0\Symp_{ct}(X_{aff}) \ar[d] \\ 
\pi_1(\scrM_{3,3}) \ar[r] & \pi_0\Symp(X)
}
\]
The relations of $G(\Gamma)$ hold in $\pi_0\Symp_{ct}(X_{aff})$, whilst the  `non-local' relations are related to passing to the projective closure.
\end{rmk}

\begin{rmk} The precise shape of the non-local relations will not matter for this paper. Indeed, 
the proof of Theorem \ref{thm:main} will actually obstruct the existence of a homomorphism $G(\Gamma) \to \Gamma_5$ lifting the homological monodromy. Note that in principle it could have been that such lifts did exist, but there were no such which satisfied the triangle / non-local conditions.
\end{rmk}

Let $Br_n$ denote the braid group on $n$ strings, with standard generators $t_1,\ldots,t_{n-1}$.

\begin{lem} \label{lem:conjugate}
The generators $\sigma_v$, for $v\in \mathrm{Vert}(\Gamma)$, are all conjugate in $G(\Gamma)$.
\end{lem}

\begin{proof}
Each edge of $\Gamma$ defines a representation $Br_3 \to G(\Gamma)$ taking the standard generators of the braid group $t_1, t_2$ to the generators associated to the vertices $v,w$ at the ends of the edge. The generators $t_i \in Br_3$ are conjugate.
\end{proof}

One can also view Lemma \ref{lem:conjugate} as a consequence of the fact that the fundamental group of the discriminant complement $\bP H^0(\mathcal{O}_{\bP^4}(3)) \backslash \Delta_{3,3}$  is normally generated by a meridian.

There is a universal family of cubics $\bP^4 \times \scrM_{3,3} \supset \scrX \stackrel{\pi}{\longrightarrow} \scrM_{3,3}$, and a local system $R^3\pi_* \bZ \to \scrM_{3,3}$ with fibre $H^3(X;\bZ) \cong \bZ^{10}$, equipped with the skew-symmetric intersection pairing. Monodromy of this local system defines a natural representation $\pi_1(\scrM_{3,3}) \to Sp(10;\bZ)$, which one can compose with the quotient map $G(\Gamma) \to \pi_1(\scrM_{3,3})$.

 A \emph{transvection} in $Sp(10;\bZ)$ is any  matrix conjugate to $\left(\begin{smallmatrix}1 & 1\\0 & 1\end{smallmatrix}\right)\oplus \mathrm{Id}_8$. Note that 
\begin{enumerate}
\item a transvection is not conjugate to its inverse or any other power of a transvection,
\item  a transvection $A$ defines a one-dimensional subspace of $\bZ^{10} \otimes \bQ$ (the `direction' of transvection), via the image of $I-A$; and 
\item the fixed locus of a transvection is a hyperplane, so meets each positive-dimensional symplectic subspace non-trivially.
\end{enumerate}

\begin{lem} \label{lem:transvection}
Under $G(\Gamma) \to Sp(10;\bZ)$, each generator $\sigma_v$ is mapped to a transvection.
\end{lem}

\begin{proof}
The generators $\sigma_v$ correspond to monodromies of families of cubics where the central fibre acquires a single ordinary double point (node) singularity. The Lagrangian sphere vanishing cycle at the node represents a non-trivial and primitive homology class in $H_3(X;\bZ)$. The result then follows from the Picard-Lefschetz formula. 
\end{proof}

\begin{lem} \label{lem:irred}
The homological monodromy $G(\Gamma) \to Sp(10;\bZ)$ is irreducible.
\end{lem}

\begin{proof}
This follows from Beauville's analysis \cite{Beauville-monodromie}.  He shows that the image of the monodromy comprises the subgroup of the symplectic group which preserves a quadratic refinement of the intersection pairing.
\end{proof}

Recall from Remark \ref{rmk:Milnor} that the singularity $\{ \sum_{i=1}^4 z_i^3 = 0\}$ has Milnor number $16$, and the 16 elements of $\{0,1\}^4$ index the vanishing cycles which form a basis of $H_3$ of the Milnor fibre. This lattice carries a skew-symmetric intersection pairing, which has kernel of rank 6; the vanishing cycles span the 10-dimensional $H_3$ of the projective closure of the Milnor fibre.  We will need the following mild strengthening of this fact.

\begin{lem} \label{lem:14_full_rank}
The sublattice of $\bZ^{10}$ spanned by the 14 non-extremal vertices has rank 10.
\end{lem}

\begin{proof}
The intersection form amongst the vanishing cycles was computed in \cite{Hefez-Lazzeri}.  This is determined by skew-symmetry, along with the following rule: if ${\bf i}$ and ${\bf j}$ belong to $\{0,1\}^4$, and ${\bf i} < {\bf j}$ in the lexicographic order, then
\[
\langle v_{{\bf i}}, v_{{\bf j}}\rangle = \begin{cases} -(-1)^{\sum_\nu (i_\nu - j_\nu)} & \mathrm{if} \  (i_\nu - j_\nu)(i_\mu - j_\mu) \geq 0 \quad \forall \, \mu,\nu \\ 0 & \mathrm{otherwise} \end{cases}
\]
(compare to the definition of $G(\Gamma)$, whose definition involves the underlying unsigned pairing). 
Direct computation shows that the submatrix indexed by the 14 non-extremal vertices has rank 10. 
\end{proof}

\begin{lem} \label{lem:Br_8}
There is a sequence $v_1,\ldots,v_7 \subset \mathrm{Vert}(\Gamma)$ and a homomorphism $Br_8 \to G(\Gamma)$ taking $t_i \mapsto \sigma_{v_i}$. In particular, for each $v\in \mathrm{Vert}(\Gamma)$, there is a homomorphism $Br_6 \to Z(\sigma_v)$  to the centraliser of $\sigma_v$, taking standard generators of the braid group to conjugates of standard generators of $G(\Gamma)$.
\end{lem}

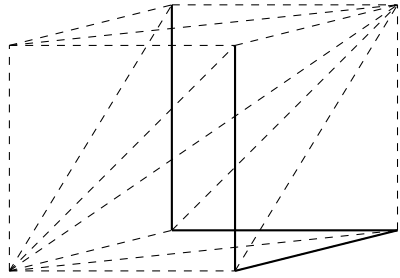
\begin{figure}[ht]
\begin{center} 
\begin{tikzpicture}[scale=0.03]

\draw[dashed] (0,100) -- (100,100);
\draw[dashed] (0,0) -- (0,100);
\draw[thick] (100,0) -- (100,100); 
\draw[dashed] (0,0) -- (100,0);

\draw[dashed] (0,0) -- (72,18);
\draw[dashed]  (100,0) -- (172,18);
\draw [dashed] (0,100) -- (72,118);
\draw[dashed]  (100,100) -- (172,118);
\draw[thick] (72,18) -- (72,118); 

\draw[dashed]  (0,0) -- (100,100); 
\draw[thick] (100,0) -- (172,18); 

\draw[dashed]  (0,0) -- (172,18); 
\draw[dashed]  (0,100) -- (172,118);
\draw[dashed]  (0,0) -- (72,118);
\draw[dashed]  (0,0) -- (172,118);

\draw[thick] (72,18) -- (172,18); 
\draw[dashed]  (72,118) -- (172,118);
\draw[dashed]  (72,18) -- (172,118);
\draw[dashed]  (172,18) -- (172,118);
\draw[dashed]   (100,0) -- (172,118);

\end{tikzpicture}
\end{center}
\caption{A copy of $Br_6$\label{Fig:2d_case_braid}}
\end{figure}

\begin{proof} An ordered sequence of vertices $v_1,\ldots,v_k$ defines a representation $Br_{k+1} \to G(\Gamma)$ exactly when the subgraph spanned by the $v_i$ inside $\Gamma$ is a linear chain, so $v_i$ is joined by an edge to $v_{i+1}$ for $1 \leq i \leq k-1$ and there are no other edges between these vertices.  One checks the following ordered sequence of vertices has this property:
\[
(0001) \to (0101) \to (0100) \to (0110) \to (0010) \to (1010) \to (1000).
\] 
The subchain spanned by the last five elements defines a copy of $Br_6$ which commutes with the generator associated to $(0001)$.   The graph from Figure \ref{Fig:2d_case} is isomorphic to the subgraph of $G(\Gamma)$ of  vertices with last co-ordinate $0$; the copy of $Br_6$ in the centraliser of $\sigma_v$ for $v=(0001)$, i.e. a chain of five vertices spanning no triangles, is indicated in Figure \ref{Fig:2d_case_braid}.
\end{proof}

\begin{lem}\label{lem:nonextremal_commute}
For each $v\neq v_{ext}$, there is some $i\in \{1,3,4,5,6,7\}$ with $[\sigma_v,\sigma_{v_i}] = 1$.
\end{lem}

\begin{proof} This is a straightforward check.
\end{proof}

The `affine braid group' $Br_{k+1}^{aff}$ is generated by elements $t_1,\ldots,t_k$ which braid cyclically, so  in place of asking that $t_1$ and $t_{k}$ commute, one instead imposes that $t_1 t_{k} t_1 = t_{k} t_1 t_{k}$.
  
\begin{lem} \label{lem:affine_braid}
The representation $Br_8 \to G(\Gamma)$  of Lemma \ref{lem:Br_8} extends to a representation $Br_9^{aff} \to G(\Gamma)$.
\end{lem}

\begin{proof}  The sequence of vertices
\[
(0001) \to (0101) \to (0100) \to (0110) \to (0010) \to (1010) \to (1000) \to (1001) \to (0001)
\]
forms a cycle in $\Gamma$ which spans no other edges.
\end{proof}

\subsection{Rigidity results for mapping class groups}

Let $\Sigma$ be an oriented surface, perhaps with boundary. We will write $\Sigma_g^b$ when we wish to specify that the  surface has genus $g$ and $b$ boundary components. We consider simple closed curves up to isotopy, and they are assumed to be homotopically non-trivial.
We say that a collection of simple closed curves $\{\gamma_1,\ldots,\gamma_r\} \subset \Sigma$ is \emph{admissible} if they are pairwise non-isotopic and may be realised simultaneously pairwise disjointly.  For a simple closed curve $\gamma \subset \Sigma$, we write $\tau_{\gamma}$ for the Dehn twist along $\gamma$.

The mapping class group $\Gamma(\Sigma) := \pi_0\Diff^+(\Sigma)$. The maps in $\Gamma(\Sigma)$ may permute boundary components, and isotopies are not required to fix the boundary. We write $\Gamma(\Sigma,C)$ for a subset $C\subset \partial \Sigma$ if we wish to consider maps which are the identity on $C$ and isotopies fixing $C$ pointwise; and we write $\Gamma(\Sigma_g^b)$ the mapping class group $\Gamma(\Sigma, \partial \Sigma)$ when $\Sigma = \Sigma_g^b$. Thurston's classification of surface diffeomorphisms \cite{FLP} yields the following trichotomy for an element $f$ of the mapping class group of $\Sigma$: 

\begin{enumerate}
\item (periodic) $f$ is isotopic to a periodic diffeomorphism;
\item (pseudo-Anosov) there is no simple closed curve $\gamma \subset \Sigma$ not isotopic to a component of $\partial \Sigma$ and positive integer $m>0$ with $f^m(\gamma) \simeq \gamma$; 
\item (reducible) $f$ preserves a non-empty admissible system of curves.
\end{enumerate}

To clarify the first case, if $\Sigma$ has non-empty boundary, then $f$ is isotopic to a periodic map through diffeomorphisms which don't preserve the boundary pointwise.  If $f$ is periodic and fixes a component $C\subset \partial\Sigma$ pointwise, then there are coprime integers $m$ and $l$ such that $f^m$ is isotopic in $\Diff(\Sigma,C)$ to the Dehn twist $\tau_C^l$ (in particular here $l\neq 0$). 

\begin{lem} \label{lem:mccarthy}
The centralizer of a pseudo-Anosov map $\psi \in \Gamma(\Sigma)$ is virtually cyclic. Each element of the centralizer is itself either pseudo-Anosov or periodic. \end{lem}

\begin{proof} See \cite{McCarthy}. \end{proof}

A curve $\gamma \subset \Sigma$ is a `reducing curve' for $f$ if $f^m(\gamma) \simeq \gamma$ for some $m\in \bZ_{>0}$, and the collection $\{\gamma, f(\gamma), f^2(\gamma),\ldots, f^{m-1}(\gamma)\}$ is admissible. 
A reducible mapping class $f$ has a \emph{canonical} `essential reduction system' of curves $\sigma(f)$, introduced in \cite{BLM}; a simple closed curve $\gamma$ belongs to $\sigma(f)$ precisely if it is a reducing curve for $f$, and if no other reducing curve for $f$ has non-trivial geometric intersection number with $\gamma$.  Note that $\sigma(f)$ is an admissible collection, i.e. the curves in $\sigma(f)$ have vanishing pairwise geometric intersection number.  

\begin{lem} \label{lem:commute_preserve_sigma}
If $[f,g]=1$ then $g(\sigma(f)) = \sigma(f)$.  Moreover, the curves of $\sigma(f)$ and $\sigma(g)$ have vanishing pairwise geometric intersection number, so have admissible union.
\end{lem}

\begin{proof} See \cite{BLM}.
\end{proof}

A simple closed curve $\gamma \subset \Sigma$ defines two (perhaps co-incident) subsurfaces $\Sigma_{l,r} \subset \Sigma$ lying to the left / right of $\gamma$.  These subsurfaces may have other boundary components (which may or may not be boundary components of $\Sigma$ itself).

\begin{prop} \label{prop:characterise_reduction_curves}
Suppose $f$ preserves the simple closed curve $\gamma$ and also preserves its orientation. Then $\gamma \subset \sigma(f)$ precisely in the following three cases:
\begin{enumerate}
\item $f$ is pseudo-Anosov on at least one of $\Sigma_l$ or $\Sigma_r$;
\item $f$ is periodic of common period $m$ on $\Sigma_l$ and $\Sigma_r$, and $f^m|_{\Sigma_l \cup \Sigma_r}$ agrees with a non-trivial power of the Dehn twist on $\gamma$;
\item $f$ is periodic of orders $m_l \neq m_r$ on $\Sigma_l$ respectively $\Sigma_r$.
\end{enumerate}
\end{prop}

\begin{proof} See \cite{Castel}. \end{proof}

Let $I(\cdot,\cdot)$ denote geometric intersection number (extended linearly to finite systems of curves in the usual way).

\begin{lem} \label{lem:braid_relations}
Let $\gamma$ and $\gamma'$ be two (non-isotopic) simple closed curves on $\Sigma$. Then
\begin{enumerate}
\item $\tau_{\gamma} \tau_{\gamma'} = \tau_{\gamma'} \tau_{\gamma}$ if $I(\gamma,\gamma') = 0$;
\item $\tau_\gamma \tau_{\gamma'} \tau_{\gamma} = \tau_{\gamma'} \tau_\gamma \tau_{\gamma'}$ if $I(\gamma,\gamma') = 1$;
\item $\langle \tau_{\gamma}, \tau_{\gamma'}\rangle$ is a free group if $I(\gamma,\gamma') \geq 2$.
\end{enumerate}
\end{lem}

\begin{proof} See \cite{Farb-Margalit}.
\end{proof}

\begin{rmk}\label{rmk:triangle}
Further constraints are imposed by the `triangle' relation \eqref{eqn:triangle_relations}. Let $\gamma, \gamma', \gamma''$ be three curves on $\Sigma$ which meet with pairwise geometric intersection number one.  One can find an essential subsurface $\Sigma \supset \Sigma_1^b \supset \{\gamma \cup \gamma' \cup \gamma''\}$ for $1 \leq b \leq 3$, i.e. one for which the boundary curves of $\Sigma_1^b$ are homotopic to boundary curves of $\Sigma$ or are not nullhomotopic in $\Sigma$; $b=1$ is the case in which $\gamma''$ is obtained by `surgery' on $\gamma \cup \gamma'$. Then
$\tau_{\gamma} \tau_{\gamma'} \tau_{\gamma''} \tau_{\gamma} = \tau_{\gamma'} \tau_{\gamma''} \tau_{\gamma} \tau_{\gamma'}
$
holds in $\Gamma(\Sigma_1^b)$ exactly when $b\in \{1,2\}$, cf. the proof of \cite[Lemma 5.10]{Salter}. 
\end{rmk}

A `multitwist'  is  a product of non-trivial powers of Dehn twists along the curves of an admissible configuration.  

\begin{lem} \label{lem:multitwists_commute}
If $T_a$ and $T_b$ are multitwists in admissible curve systems $a,b$ and $I(a,b) \neq 0$ then $[T_a,T_b] \neq 1$.
\end{lem}

\begin{proof} The essential reduction system of a multitwist is the set of curves in which one is twisting \cite{BLM}. 
 If the maps commute, $T_a$  preserves $b$ and vice-versa, by Lemma \ref{lem:commute_preserve_sigma}. \end{proof}

A \emph{chain of curves} in $\Sigma$ is a sequence $\gamma_1,\ldots, \gamma_l$ for which the geometric intersection numbers satisfy
\[
I(\gamma_i, \gamma_{i+1}) = 1, \ 1\leq i\leq l-1, \quad I(\gamma_i,\gamma_j) = 0 \ \mathrm{if} \ |i-j|>1.
\]
Note that the $\gamma_i$ are necessarily non-separating; a chain of length $l$ exists on a surface of genus $g$ only when $g \geq l/2-1$.  Lemma \ref{lem:braid_relations} shows that a chain of length $l$ on $\Sigma$ defines a representation $Br_{l+1}\to \Gamma(\Sigma)$.  

\begin{thm}[Castel]\label{thm:castel}
Let $n\geq 6$ and $\rho: Br_n \to \Gamma(\Sigma_g^b)$ be a homomorphism, with $g\leq n/2$.  Then $\rho$ has cyclic image, or there is a chain $\gamma_1,\ldots,\gamma_{n-1}$ of simple closed curves, a fixed sign $\varepsilon \in \{-1,+1\}$, and an element $w \in \cap_j \, Z(\tau_{\gamma_j})$, such that $\rho(t_i) = \tau_{\gamma_i}^{\varepsilon}\cdot w$ for each $i$. 
\end{thm}

The second case of Theorem \ref{thm:castel} can only arise if a chain exists, so when $g\geq n/2-1$.

The hypotheses in Theorem \ref{thm:castel} are remarkably weak ($\rho$ is not assumed to be injective, $b$ is not constrained). The common centralizer $\cap_j Z(\tau_{\gamma_j})$ is generated by maps supported on the complement of the chain, and by the hyperelliptic involution in the subsurface neighbourhood of the chain. 

Castel speculates  that the result should hold whenever $g \leq n-2$.  In this paper, we will encounter the case $n=8, g=5$ which falls outside the known results but within the expected ones; however, we will know something about the composite homomorphism $Br_8 \to \Gamma_5 \to Sp(10;\bZ)$, which gives us leverage not available in the general setting.

\section{Constraints on essential reduction systems}

Suppose for contradiction that  \eqref{eqn:monodromy} lifts, i.e. that there is a factorisation of the homomorphism 
\[
G(\Gamma)\to \pi_1(\scrM_{3,3}) \to  \Gamma(\Sigma) \to Sp(10;\bZ)
\]
where $\Sigma = \sqcup_i \Sigma_{g_i}$ with $g_i > 0$ for each $i$ and $g_1 + \cdots + g_r = 5$, and where the composite recovers the usual cohomological monodromy of the universal family.  For most of the paper we focus on the case where  $r=1$ and the monodromy factors through $\Gamma_5$; the other (easier) cases are considered in Corollary \ref{cor:disconnected}.    We write $\rho: G(\Gamma) \to \Gamma_5$ for the resulting homomorphism to the mapping class group, and will write $\rho(\sigma_v) = f_v \in \Gamma_5$, for $v\in \mathrm{Vert}(\Gamma)$.

\begin{lem} Each $f_v$ is a reducible mapping class. \end{lem}

\begin{proof} Since all the $\sigma_v$ are conjugate, they all have the same Thurston type. Lemma \ref{lem:Br_8} and Lemma \ref{lem:mccarthy}   together rule out $f_v$ being pseudo-Anosov. A periodic non-identity map of a closed surface necessarily has eigenvalues which are non-trivial roots of unity, which would contradict Lemma \ref{lem:transvection}.
\end{proof}

By the structure theorem, $f_v$ has an essential reduction system $\sigma(f_v)$, a finite union of pairwise disjoint and pairwise non-isotopic simple closed curves, the union of which is invariant under $f_v$. Our aim is to prove that the essential reduction system of $f_v$ comprises a single reduction curve.

\subsection{The case of a single essential reduction curve} 

\begin{lem} \label{lem:most_of_the_point}
Suppose that $f_v$ has a unique essential reduction curve $\gamma$. Then $\gamma$ is non-separating, and $f_v$ is a positive Dehn twist in $\gamma$. 
\end{lem}

\begin{proof}
Suppose for contradiction that the unique essential reduction curve separates.  Since $f_v$ acts preserving $\sigma(f_v) = \{\gamma\}$, and $\gamma$ separates $\Sigma$ into subsurfaces of different genera, $f_v$ must preserve the two subsurfaces and preserve the orientation of $\gamma$.  From Proposition \ref{prop:characterise_reduction_curves}, one sees that one of the following (not necessarily mutually exclusive) three cases must occur:
\begin{enumerate}
\item $f_v = \tau_{\gamma}^k$ is a power of a Dehn twist on a separating curve;
\item $f_v$ acts as a periodic map with non-zero period $m$ on a subsurface of positive genus;
\item $f_v$ acts by a pseudo-Anosov map on a subsurface of positive genus.
\end{enumerate} 
Both the first two cases are incompatible with the action of $f_v$ on homology; in the first case it would act trivially, in the second case it would act by a matrix whose eigenvalues included some non-trivial $m$-th root of unity. 

So we have a subsurface $\Sigma_\psi \subset \Sigma$ on which $f_v$ acts by a pseudo-Anosov map $\psi$. Moreover, the complementary subsurface $\Sigma_{\neq\psi} \subset \Sigma$ contains no essential reduction curve, so the induced map on this subsurface is irreducible and hence either periodic or pseudo-Anosov. The former case is again ruled out for homological (eigenvalues of action on $H_1(\Sigma;\bZ)$) reasons unless the periodic map is trivial.  We conclude that if there is a unique reduction curve which furthermore separates, then $f_v$ is either pseudo-Anosov on both subsurfaces, or pseudo-Anosov on one and trivial on the other.

Now let $g\in Br_6 \to Z(f_v)$ be an element of the braid group mapping to the centraliser of $f_v$ as in Lemma \ref{lem:Br_8}. Then $g(\sigma(f_v)) = \sigma(f_v)$ so $g$ preserves $\gamma$, and hence the corresponding subsurfaces.  In particular, $g$ induces a mapping class on $\Sigma_{\psi}$ which commutes with $\psi$ on $\Sigma_\psi$.  By Lemma \ref{lem:mccarthy}, we obtain a virtually cyclic representation $\rho: Br_6 \to \Gamma(\Sigma_{\psi})$ for which each non-trivial element of the centraliser of $\psi$ is either pseudo-Anosov or periodic. In both cases, the curve $\gamma$ belongs to the essential reduction system of the centralising element.

If $f_v$ is pseudo-Anosov on both $\Sigma_{\psi}$ and $\Sigma_{\neq\psi}$ then we find that $Br_6 \to \Gamma(\Sigma)$ preserves $\gamma$ and acts virtually cyclically on both subsurfaces. This contradicts the homological monodromy.  Therefore, $f_{v_1}$ acts by $\psi$ on one subsurface (say of genus $a>0$) and trivially on the other (of genus $5-a$).

Since $f_{v_j}$ is conjugate to $f_{v_1}$, it acts with a unique essential reduction curve, which separates $\Sigma$ into a subsurface $\Sigma_a^1$ on which $f_{v_j}$ acts by a pseudo-Anosov, and a subsurface $\Sigma_{5-a}^1$ on which $f_{v_j}$ acts trivially.  Since $\gamma \subset \sigma(f_{v_j})$, we must have equality, and since $a\neq 5-a$, $f_{v_j}$ acts trivially on the subsurface on which $f_{v_1}$ acted trivially.  Therefore, the representation $Br_6 \to Sp(10;\bZ)$ has virtually cyclic image, a contradiction.

Finally, suppose $\gamma$ is indeed non-separating.  By Proposition \ref{prop:characterise_reduction_curves}, the map $f_v$ on $\Sigma \backslash \gamma$ is either pseudo-Anosov, or has period $>1$, or $f_v = \tau_{\gamma}^{k}$ is a power of the Dehn twist along $\gamma$.  The first two cases lead to contradictions as before, either by considering the induced cyclic action of a centralising $Br_6$, or eigenvalues of the map on homology.  This reduces us to $f_v$ acting by the $k$-th power of a non-separating Dehn twist, and then the homological action forces  $k=1$.   
\end{proof}

\subsection{Separating essential reduction curves are fixed}

\begin{lem} \label{lem:separating_fixed} Suppose that $\sigma(f_v)$ contains a separating reduction curve $\gamma^{sep}$. Then this curve, and (either choice of) its orientation, are preserved by $f_v$.
\end{lem}

\begin{proof}
Consider the collection of separating essential reduction curves $\gamma^{sep} = \gamma_1^{sep}, \gamma_2^{sep}, \ldots, \gamma_r^{sep}$ for $f_v$.  Since $f_v$ preserves $\sigma(f_v)$ it induces a permutation of these $r$ curves.  Each of the separating curves splits $\Sigma$ into two components, either of genus one and four, or of genus two and three; the permutation of the $\gamma_j^{sep}$ induced by $f_v$ must preserve the splitting type.  If (relabelling indices if necessary) $f_v(\gamma_1^{sep}) = \gamma_2^{sep}$, then the smaller genus subsurface bound by the first curve must be taken diffeomorphically to the smaller genus subsurface bound by the second. Since $\gamma_1^{sep} \cap \gamma_2^{sep} = \emptyset$, these subsurfaces have disjoint interiors. 

 The symplectic vector space $H_1(\Sigma;\bQ)$ is split into orthogonal non-trivial symplectic subspaces indexed by the components of  $\Sigma \backslash \cup_j \gamma_j^{sep}$, so $f_v$ induces a non-trivial permutation of the corresponding block decomposition of $H_1(\Sigma;\bQ)$. This is not compatible with the fact that $f_v$ acts by a transvection, so has fixed locus a hyperplane for the action on $H_1(\Sigma;\bQ)$ (a hyperplane  meets every symplectic subspace non-trivially). Therefore, $f_v$ preserves each $\gamma_j^{sep}$ setwise.
 
 Since $f_v$ preserves $\gamma^{sep}$, it must preserve the two complementary subsurfaces (which have different genera), so it preserves the orientation of $\gamma^{sep}$. 
\end{proof}

 Let $\bar\rho: Br_6 \to Z(f_v)$ be the homomorphism from Lemma \ref{lem:Br_8}  to the centraliser of $f_v$.  An element  $\eta \in Br_6$ defines a mapping class $\bar\rho(\eta):= f_{\eta}$ which stabilises $\sigma(f_v)$.  The $f_v$-orbit of $\gamma_1^{sep}$ inside $\sigma(f_v)$ comprises at most 5 curves in the genus 1+4 splitting case and at most two in the genus 2+3 splitting case.  When $k<n$, any homomorphism from $Br_n$ to $\mathrm{Sym}_k$ has cyclic image \cite{Artin}, so all the elements $f_{\eta}$ induce the same permutation of the separating essential reduction curves for $f_v$ and of their complementary subsurfaces.  As usual, this general statement can be strengthened by using our  knowledge of the underlying symplectic representation. 
 
\begin{lem} \label{lem:unmoved}
The subgroup $\bar\rho(Br_6)$ induces the trivial permutation of $\{\gamma_j^{sep}\}$.
\end{lem}

\begin{proof} 
Consider a standard generator $f_v$ of $\bar\rho(Br_6)$. If $\{\gamma_j^{sep}\}$ is permuted  non-trivially by $f_v$, the subsurfaces bound by the elements of the orbit are also permuted non-trivially. One then obtains a contradiction as in the proof of Lemma \ref{lem:separating_fixed}, using the fact that $f_v$ acts by a transvection so its fixed locus on cohomology meets every symplectic subspace non-trivially.
\end{proof}

Thus $\bar\rho$ induces an action of $Br_6$ on each of the two complementary subsurfaces to $\gamma^{sep}$, i.e. fixing a particular separating reduction curve $\gamma^{sep}$, we have induced representations $\bar\rho_i: Br_6 \to \Gamma(\Sigma_i^1)$ for $i\in \{1,4\}$ or $i\in \{2,3\}$.  For definiteness, we will consider the copy of $Br_6$ associated to the centraliser of $f_{v_1}$ and with generators labelled by $t_3,\ldots, t_7$, mapping to $f_{v_i}$ with $3\leq i\leq 7$.

\subsection{Genus $(2,3)$-separating essential reduction curves}

Suppose there is a separating essential reduction curve for $f_{v_1}$ which separates $\Sigma$ into pieces $\Sigma_2^1 \cup \Sigma_3^1$ of genus 2 and 3.  The representations $\bar\rho_2$ and $\bar\rho_3$ are governed by Theorem \ref{thm:castel}.

\begin{lem} \label{lem:force_cyclic} $\bar\rho_2$ is cyclic.
\end{lem}

\begin{proof} If not, the action is associated to a chain of curves $a_3,\ldots, a_7$ with $f_{v_i} = \tau_{a_i}^{\varepsilon}\cdot w$ for some centralising element $w$ and $\varepsilon \in \{\pm1\}$.  The elements of the chain of five curves $a_i$ on the genus two surface necessarily satisfy a homology relation.   However, in the homological monodromy, the representation $Br_6 \to Sp(10;\bZ)$ extends to a representation $Br_8 \to Sp(10;\bZ)$.  That forces the homology classes $[a_i] \in \bZ^{10}$ supporting the transvections, for $3\leq i\leq 7$, to be pairwise linearly independent.
\end{proof}

Clearly it follows that $\bar\rho_3: Br_6 \to \Gamma(\Sigma_3^1)$ is not cyclic, hence is associated to a chain of curves $\{a_3,\ldots,a_7\}$ in $\Sigma_3^1$. 

\begin{lem} 
The representation $\bar\rho_3: Br_6 \to \Gamma(\Sigma_3^1)$ extends to a representation $\bar\rho_4: Br_8 \to \Gamma(\Sigma_3^1)$.
\end{lem}

\begin{proof}
Since $\gamma^{sep}$ is a separating essential reduction curve of $f_{v_1}$, it must be that $f_{v_i}$ has some separating essential reduction curve of separating type $(2,3)$ for each $3\leq i\leq 7$. Since $[f_{v_1},f_{v_i}] = 1$, the corresponding curve for $f_{v_i}$ is disjoint from $\gamma^{sep}$.  It cannot lie inside $\Sigma_2^1$ and separate off a genus two subsurface, and it cannot lie inside $\Sigma_3^1$ because of the known structure of $\bar\rho_3$. It follows that the separating reduction curve of $f_{v_i}$ must be $\gamma^{sep}$.  The argument also implies that this is the unique separating reduction curve of type $(2,3)$ for $f_{v_3}$, and hence each $f_{v_i}$ has a unique such curve.  

Now consider $f_{v_2}$. This commutes with $Br_5 = \langle f_{v_j} \, | \, 4 \leq j\leq 7\rangle$, so $I(\sigma(f_{v_2}), \sigma(f_{v_j})) = 0$ for $4 \leq j \leq 7$ by Lemma \ref{lem:commute_preserve_sigma}.  In particular, $\sigma(f_{v_2})$ is disjoint from $a_4,\ldots,a_7$ (which are essential reduction curves of the corresponding $f_{v_j}$, which Dehn twist non-trivially in them) and also from $\gamma^{sep}$. Cutting $\Sigma_3^1$ along $a_7\cup a_5$ yields a surface of genus $1$ with six boundary components.  It follows that the separating $(2,3)$-reduction curve of $f_{v_2}$ cannot lie inside either this surface or inside $\Sigma_2^1$, unless in the latter case it is boundary parallel, so it must also be $\gamma^{sep}$.  

We therefore have that each generator of $Br_8$ preserves $\gamma^{sep}$, and hence there is a homomorphism $Br_8 \to \Gamma(\Sigma_3^1)$. 
\end{proof}

\begin{lem}
$\bar\rho_4(t_i) = \tau_{a_i} \tau_{\partial}^k$ where $\partial = \partial \Sigma_3^1$ and $k\in \bZ$.
\end{lem}

\begin{proof}
The representation $\bar\rho_4$ is governed by Theorem \ref{thm:castel}. Thus  $\bar\rho_4(t_i) = \tau_{a_i}^{\varepsilon}\cdot w$ for $\{a_1,\ldots,a_7\}$ a chain of curves. 
The centralising element $w$ is a diffeomorphism of the complementary surface to the $A_7$-chain; but this is just the boundary annulus (and a disc, with trivial mapping class group), so $w$ is a power of $\tau_\partial$. The sign $\varepsilon = +1$ since a transvection is not conjugate to its inverse.
\end{proof}

\begin{lem} \label{lem:preserved}
If $v \neq v_{ext}$ is not an extremal vertex, $f_v$ preserves $\gamma^{sep}$.
\end{lem}

\begin{proof}
By Lemma \ref{lem:nonextremal_commute},  there is $i \in \{1, 3, 4, 5, 6, 7\}$ for which $[\sigma_v,\sigma_{v_i}] = 1 \in G(\Gamma)$.  Therefore, $f_v$ preserves $\sigma(f_{v_i})$. Since  $\gamma^{sep}$ is the unique essential reduction curve of the relevant topological type for $f_{v_i}$, we must have that $f_v$ preserves $\gamma^{sep}$.
\end{proof}

\begin{cor} \label{cor:no_23}
There cannot be a $(2,3)$-separating essential reduction curve.
\end{cor}

\begin{proof}
If such a curve exists, it is preserved by all non-extremal vertices by Lemma \ref{lem:preserved}.  This means that the subgroup of $G(\Gamma)$ generated by the non-extremal vertices acts reducibly on $H_1(\Sigma;\bZ)$, i.e. preserving a non-trivial symplectic splitting associated to the decomposition $\Sigma = \Sigma_2^1 \cup \Sigma_3^1$. This contradicts Lemma \ref{lem:14_full_rank}.
\end{proof}

\subsection{Genus $(1,4)$-separating essential reduction curves}

Suppose $\sigma(f_{v_1})$ contains a separating curve yielding the other decomposition $\Sigma = \Sigma_1^1 \cup \Sigma_4^1$. There are now induced representations $\bar\rho_1: Br_6 \to \Sigma_1^1$ and $\bar\rho_4: Br_6 \to \Sigma_4^1$.

\begin{lem}
The homomorphism $\bar\rho_1: Br_6 \to \Gamma(\Sigma_1^1)$ has cyclic image, generated by an element $\kappa$ which is a power of the Dehn twist in $\gamma^{sep}$.
\end{lem}

\begin{proof} Theorem \ref{thm:castel} implies the cyclicity. Let $\kappa$ denote a generator of the cyclic group image.

The group $\Gamma(\Sigma_1^1)$ is the universal $\bZ$-central extension of $SL(2;\bZ)$. The generators of $Br_6$ define mapping classes which act on $H_1(\Sigma;\bZ)$, and hence $H_1(\Sigma_1^1;\bZ)$, with all eigenvalues equal to $1$, which means $\kappa$ is a multitwist -- the composition of some power of the Dehn twist in $\gamma^{sep}$ with some power of the Dehn twist in a disjoint curve $\gamma^{int} \subset \Sigma^1_1$ in the interior of the handle. Suppose $\bar\rho_1(t_i)$ involves the Dehn twist in the curve $\gamma^{int}$. This is homologically essential, so to be conjugate to a transvection, $\bar\rho_4(t_i)$ must act trivially on $H_1(\Sigma_4^1;\bZ)$. However, then the transvections $\rho(t_i)$ of $H_1(\Sigma;\bZ)$   would be equal for distinct commuting generators $t_3,t_5$ of $Br_6$, since the relevant curves $\gamma^{int}$ would have to co-incide to have the commutativity relation $[\bar\rho_1(t_3), \bar\rho_1(t_5)]=1$ hold.  This is not compatible with the homological monodromy representation.  Therefore, the action on $H_1(\Sigma_1^1)$ must be trivial.
\end{proof}

\begin{lem} \label{lem:only_one_14}
There cannot be two or more $(1,4)$-separating reduction curves.
\end{lem}

\begin{proof}
Suppose $f_{v_1}$ has two separating reduction curves $\gamma_1^{sep} \sqcup \gamma_2^{sep}$ each splitting off a copy of $\Sigma_1^1$. Then there is an induced representation $Br_6 \to \Gamma(\Sigma_3^2)$ to the complementary surface, which is governed by Theorem \ref{thm:castel}, so defined by an $A_5$-chain $\{a_3,\ldots,a_7\}$ and a centralising element (which must be a boundary multitwist).  

We can now argue as in the previous subsection. The $\gamma_j^{sep}$, $j=1,2$, are preserved by $f_{v_i}$ for $3\leq i\leq 7$, whilst $\sigma(f_{v_i}) \cap \Sigma_3^2 \subset a_i \cup \partial \Sigma_3^2$; so the only possibility for the corresponding $(1,4)$-separating essential reduction curves for $f_{v_i}$ are the same $\gamma_j^{sep}$ as arose for $f_{v_1}$.  Moreover, the $\gamma_j^{sep}$ are preserved by $f_{v_2}$, so are reduction curves for this map, so have trivial geometric interection with its essential reduction curves. Since $[f_{v_2}, f_{v_i}] = 1$ for $4\leq i\leq 7$, the essential reduction curves of $f_{v_2}$ are also disjoint from the $A_4$-chain $a_4 \cup \ldots a_7$. Cutting along $a_5 \cup a_7$ leaves a surface $\Sigma_1^6$, and this cannot contain two inequivalent curves each of which bound  $\Sigma_1^1$-subsurfaces with disjoint interiors. Therefore, at least one of $\gamma_1^{sep}$ and $\gamma_2^{sep}$ must also be one of the essential reduction curves of $f_{v_2}$. 

Relabel so this is $\gamma_1^{sep}$. Then all of $\{f_{v_1},\ldots,f_{v_7}\}$ preserve $\gamma_1^{sep}$, so we get a representation $\bar\rho_4: Br_8 \to \Gamma_4^1$, to which we can apply Theorem \ref{thm:castel}. It follows that the only essential reduction curve of $f_{v_1}$ in $\Sigma_4^1$ is the boundary $\gamma_1^{sep}$ and the supporting curve $a_1$ of the transvection, in particular, there were not two separating $(1,4)$-essential reduction curves. 
\end{proof}

We return to the assumption that there is a (necessarily unique) $(1,4)$-separating essential reduction curve $\gamma^{sep}$ for $f_{v_1}$. 

\begin{lem} $f_{v_1}$ must have a homologically non-trivial essential reduction curve in $\Sigma_4^1$.
\end{lem}

\begin{proof}
Consider the case in which $\sigma(f_{v_1}) \cap \mathrm{int}(\Sigma_4^1) = \emptyset$. We know that $\gamma^{sep}$ cannot be the unique essential reduction curve for $f_{v_1}$ by Lemma \ref{lem:most_of_the_point}, so there is an essential reduction curve $\gamma^{int}$ in the interior of $\Sigma_1^1$. This must be homologically essential, by Lemma \ref{lem:only_one_14} and Corollary \ref{cor:no_23}. Then $f_{v_1}$ must act on $\Sigma_1^1$ by a non-trivial power of a Dehn twist in $\gamma^{int}$ composed with a power of the twist in $\gamma^{sep}$. The action of $f_{v_1}$ on $\Sigma_4^1$ would have to be a Torelli pseudo-Anosov map, if non-trivial, which would then force $Br_6$ to act on this subsurface virtually cyclically by an argument analogous to that of Lemma \ref{lem:most_of_the_point}. This would be a contradiction.  
\end{proof}

\begin{lem} \label{lem:14_two_cases}
If $f_{v_1}$ has a $(1,4)$-separating essential reduction curve, then this is the only separating essential reduction curve. Furtherore, 
\begin{itemize}
\item either $f_{v_1}$ has a unique non-separating essential reduction curve $\gamma^{int}$, which lies in $\Sigma_4^1$; there is an $A_5$-chain in $\Sigma_3^3 = \Sigma_4^1 \backslash \gamma^{int}$ governing the $f_{v_i}$ for $3\leq i\leq 7$;
\item or $f_{v_1}$ has $2$ non-separating essential reduction curves  $\gamma^{int}_1 \cup \gamma^{int}_2 \subset \Sigma_4^1$ whose union bounds a subsurface $\Sigma_3^2 \subset \Sigma_4^1$ separating off a pair of pants; this subsurface contains an $A_5$-chain governing the $f_{v_i}$ for $3\leq i\leq 7$.
\end{itemize}
\end{lem}

\begin{proof}
We know $f_{v_1}$ has an essential reduction curve in the interior of $\Sigma_4^1$,  necessarily homologically essential. Let $\gamma_2,\ldots,\gamma_r$ denote the components of $\sigma(f_{v_1}) \cap \Sigma_4^1$. If $\cup_{j=2}^r \, \gamma_j$ does not separate $\Sigma_4^1$, then $Br_6$ acts on the connected surface which is its complement.  This action cannot be cyclic, for homological considerations. The complement must have genus $\geq 2$ by Theorem \ref{thm:castel}, but the genus two case yields a contradiction as in the proof of Lemma \ref{lem:force_cyclic}.   Therefore $r=2$ and there was a unique essential reduction curve in the interior of $\Sigma_4^1$.  Then $Br_6$ acts via a chain of curves on $\Sigma_3^3 = \Sigma_4^1 \backslash \gamma_2$.  For the homological action to be a transvection, rather than a product of transvections with fixed locus of higher codimension, there cannot be a further essential reduction curve in $\Sigma_1^1$. 

Alternatively, $\sigma(f_{v_1}) \cap \Sigma_4^1$ separates.  Note that since the $Br_6$-action on $\Sigma_4^1$ preserves the boundary, it must preserve the component $\Sigma_{adj} \subset \Sigma_4^1 \backslash \sigma(f_v)$ adjacent to the boundary. This component has at least two further boundary components, since the complementary surfaces of the essential reduction curves have Euler characteristic $\leq -1$, and all further boundaries are homologically essential in $\Sigma$.  The Euler characteristic of $\Sigma_4^1 \backslash \Sigma_{adj}$ is $\geq -6$ so this has genus $\leq 3$; this must therefore have genus $3$, as in the previous paragraph.  So $Br_6$ acts via a chain of curves on a genus 3 subsurface.  Since the Euler characteristic of $\Sigma_4^1$ is $-7$ and any subsurface bound by $\sigma(f_{v_1})$ has $\chi \leq -1$, one must have that $r=3$, with a bounding pair of homologous essential reduction curves $\gamma_{\pm}$ splitting $\Sigma_4^1$ into a pair of pants and a copy of $\Sigma_3^2$.  
\end{proof}

\begin{lem}
The second case of Lemma \ref{lem:14_two_cases} cannot occur.
\end{lem}

\begin{proof}
Suppose $\sigma(f_{v_1}) = \{\gamma^{sep}, \gamma_+, \gamma_-\}$ for a bounding pair $\gamma_{\pm} \subset \Sigma_4^1$. Consider  $f_{v_3}$.  It acts on $\Sigma_3^2$ through an $A_5$-chain, so its only possible essential reduction curves in this subsurface are the relevant element $a_3$ of that chain, and the boundary curves $\gamma_{\pm}$. It also has $\gamma^{sep}$ as essential reduction curve, and a priori could have a further reduction curve $\sigma$ inside the complementary $\Sigma_1^1$.  However, the map would then necessarily involve some power of a Dehn twist along $\sigma$, incompatible with the homological action since we have a transvection in $a_3$.  We conclude that $f_{v_3}$ is given by a multitwist
\begin{equation} \label{eqn:first_expression}
f_{v_3} = \tau_{a_1} \circ \tau_{\gamma_+}^{k_+} \circ\tau_{\gamma_-}^{k_-} \circ \tau_{\gamma_{sep}}^k
\end{equation}
for some $k_{\pm} \in \bZ$  and $k\neq 0$.  However, since the essential reduction curves of $f_{v_1}$ are $\gamma^{sep} \cup \gamma_{\pm}$ and it cannot have a pseudo-Anosov component or periodic component for the usual homological reasons,  we know
\begin{equation} \label{eqn:second_expression}
f_{v_1} =  \tau_{\gamma_+}^{u_+} \circ\tau_{\gamma_-}^{u_-} \circ \tau_{\gamma^{sep}}^u 
\end{equation}
with $u, u_+, u_- \neq 0$.  The shapes of \eqref{eqn:first_expression} and \eqref{eqn:second_expression} cannot be made equal by choosing the powers, noting that $\gamma_+ \cup \gamma_-$ bounds but $\gamma_{\pm} \cup a_1$ does not.  This contradicts $f_{v_1}$ and $f_{v_3}$ being conjugate.
\end{proof}

\begin{cor} \label{cor:no_14}
There cannot be a $(1,4)$-separating essential reduction curve.
\end{cor}

\begin{proof}
Suppose such a curve $\gamma^{sep}$ exists.  From the previous analysis, $\gamma^{sep}$ splits $\Sigma = \Sigma_4^1 \cup \Sigma_1^1$, and $\sigma(f_{v_1}) = \{a_1 \cup \gamma^{sep}\}$ for a homologically non-trivial curve $a_1 \subset \Sigma_4^1$. Furthermore, $f_{v_1} = \tau_{a_1} \circ \tau_{\gamma^{sep}}^k$ for some $k\in \bZ$.

We now use Lemma \ref{lem:nonextremal_commute} again. For every $v \neq v_{ext}$,  $f_v$ preserves $\sigma(f_{v_i})$ for some $i \in \{1,3,\ldots,8\}$, and $\sigma(f_v)$ and $\sigma(f_{v_i})$ have trivial geometric intersection number.  In particular, $f_v$ preserves $\gamma^{sep}$, so $\sigma(f_v)$ is disjoint from $\gamma^{sep}$, and the homologically essential component $a_v \subset \sigma(f_v)$ lies wholly inside $\Sigma_4^1$ or $\Sigma_1^1$. Furthermore, for each $v\neq v_{ext}$ other than $v_1$, there is some $j$ with $3 \leq j \leq 7$ with the property that $f_v$ braids with $f_{v_j}$.  The corresponding braiding in the homological monodromy would not be compatible with $a_v \subset \Sigma_1^1$, since $a_j \subset \Sigma_4^1$ for $3\leq j\leq 7$.

It follows that there is a fixed subsurface $\Sigma_4^1 \subset \Sigma$ with the property that for all non-extremal vertices $v$, the map $f_v$ acts by a transvection in a class supported in $\Sigma_4^1$. This contradicts Lemma \ref{lem:14_full_rank}. 
 \end{proof}

\begin{cor}
All essential reduction curves for $f_v$ are non-separating.
\end{cor}

\begin{proof} Immediate from Corollary \ref{cor:no_23} and \ref{cor:no_14}.
\end{proof}

\subsection{All essential reduction curves are non-separating}

\begin{cor} If $G(\Gamma) \to Sp(10;\bZ)$ lifts to $\Gamma_5$, then necessarily each vertex generator is sent to a Dehn twist in a non-separating simple closed curve.
\end{cor}

\begin{proof}
 Fix one essential reduction curve $\gamma_1$ of $f_{v_1}$ and consider its $f_{v_1}$-orbit $\{\gamma_1,\ldots,\gamma_r\} \subset \sigma(f_{v_1})$, on which $f_{v_1}$ acts by a cyclic permutation.  If these curves span a subspace of rank $>1$ in $H_1(\Sigma;\bQ)$ then $f_v$ must have a non-trivial root of unity in its spectrum, a contradiction. On the other hand, three distinct pairwise disjoint homologically essential curves necessarily span at least a two-dimensional subspace of homology.  We conclude that either $r=2$ and $\gamma_1 \cup \gamma_2$ separates $\Sigma$, or the essential reduction curve $\gamma_1$ is preserved by $f_{v_1}$.

In the first case, we have an invariant pair $\gamma_1\cup \gamma_2$ of essential reduction curves which separate $\Sigma$ into subsurfaces $\Sigma_a^2 \cup \Sigma_{4-a}^2$ with $a\geq 1$, so both subsurfaces have genus $\leq 3$ and are bound by Theorem \ref{thm:castel}. (If $a=2$, although the subsurfaces are homeomorphic, they cannot be permuted by $f_v$ for the usual reason that then the homology action of the map would permute non-trivial symplectic subspaces of $H_1(\Sigma;\bQ)$, contradicting the fact that the fixed hyperplane meets each positive-dimensional symplectic subspace non-trivially.)  Arguing as in previous cases, we again find that the generators of the $Br_6$-action act by multitwists in this case, with a positive twist along a homologically non-trivial curve, perhaps composed with  separating Dehn twists or their inverses in boundary components.  As in the proof of Corollary \ref{cor:no_14}, we have $\gamma_1 \cup \gamma_2 \subset \sigma(f_{v_j})$ for $3\leq j\leq 7$, and the preservation of these curves by $f_v$ for all non-extremal $v$ yields a contradiction to the homological monodromy representation.

In the second case, $f_v$ acts on $\Sigma \backslash \gamma_1 = \Sigma_4^2$.  If $\gamma_1$ was the unique essential reduction curve, we are done by Lemma \ref{lem:most_of_the_point}. If not, we iterate the argument: another reduction curve is either preserved setwise, or permuted with a bounding sibling, and we can cut to reduce the geometric complexity of the surface, and find a set of curves preserved by the $Br_6$ action and hence by all maps commuting with that subgroup, violating Lemma \ref{lem:14_full_rank}. 
\end{proof}

\section{Impossible curve configurations}

Suppose the homological monodromy $G(\Gamma) \to Sp(10;\bZ)$ factors through $\Gamma_5$. Then each $f_v$ is given by the positive Dehn twist in a homologically non-trivial simple closed curve $\gamma_v \subset \Sigma_5$. Since curve configurations can be `pulled tight', we can assume that each pair of curves $\gamma_v,\gamma_{v'}$ meet in exactly their geometric intersection number $I(\gamma_v, \gamma_{v'})$ of points.  The geometric intersection pattern of these simple closed curves is therefore constrained by Lemma \ref{lem:braid_relations}  as follows:
\begin{enumerate}
\item if $v,v'$ do not belong to an edge of $\Gamma$, the corresponding curves are disjoint;
\item if $v,v'$ do belong to an edge of $\Gamma$, the corresponding curves meet in one point.
\end{enumerate}

In particular, the seven closed curves defining the morphism $Br_8 \to G(\Gamma)$ from Lemma \ref{lem:Br_8} yield an $A_7$-chain in $\Sigma_5$, which moreover extends to an $A_8$-loop corresponding to the $Br_9^{aff}$-extension of Lemma \ref{lem:affine_braid}.  Chains of simple closed curves are unique up to diffeomorphism after specifying that their homology classes are linearly independent. More precisely,  the subsurface neighbourhood of an $A_7$-chain is  a $\Sigma_3^2$,  but there are 3 different models for the inclusion $\Sigma_3^2 \subset \Sigma_5$:
\begin{itemize}
\item the two boundary components may separate off a copy of $\Sigma_1^2$;
\item each boundary component may separate off a copy of $\Sigma_1^1$;
\item one boundary component may separate a disc, and the other a  $\Sigma_2^1$.
\end{itemize}
The first situation is distinguished by the $7$ curves of the chain being homologically linearly independent, so $a_1+a_3+a_5+a_7 \neq 0 \in H_1(\Sigma_5;\bZ)$. 

\begin{lem}
The curves $a_j$ defining the $A_7$-chain satisfy $a_1+a_3+a_5+a_7 \neq 0$.
\end{lem}

\begin{proof}
In the lexicographic notation for the vertices $\{0,1\}^4$, the chain of vertices $v_1,\ldots,v_7$ defining the braid group $Br_8$ is given by
\[
(0001) - (0101) - (0100) - (0110) - (0010) - (1010) - (1000) \qquad \simeq \qquad v_1 - v_2 - \cdots - v_7.
\]
Consider the non-extremal vertex $\hat{v}$ associated to $(0111)$. The associated symplectic matrix $A_{\hat{v}}$ is a transvection whose defining class $a_{\hat{v}}$ in $\bZ^{10}$ has intersection number $\pm 1$ with $a_i$ for $1\leq i\leq 5$, since it bounds an edge in the Artin graph $\Gamma$ with those vertices and the corresponding transvections should braid, but intersection number zero with $a_6,a_7$, since there are no edges between $(10\bullet\bullet)$ and $(01\bullet\bullet)$ for any values $\bullet$.  Therefore $\langle a_{\hat{v}}, (a_1+a_3+a_5+a_7) \rangle \neq 0 $ modulo $2$.
\end{proof}

The $A_7$-chain therefore looks as in Figure \ref{Fig:A7_chain}, where we have relabelled the supporting curves $\{a,b,\ldots,g\}$. Thus,
\[
\xymatrix{
(0001) \ar@{-}[r] \ar@{-}[d] & (0101) \ar@{-}[r] \ar@{-}[d] & (0100) \ar@{-}[r] \ar@{-}[d] & (0110) \ar@{-}[r] \ar@{-}[d] & (0010) \ar@{-}[r] \ar@{-}[d] & (1010) \ar@{-}[r] \ar@{-}[d] & (1000) \ar@{-}[d] \\
a  \ar@{-}[r] & b \ar@{-}[r]  & c \ar@{-}[r] & d \ar@{-}[r] &e  \ar@{-}[r] & f \ar@{-}[r] & g
}
\]  
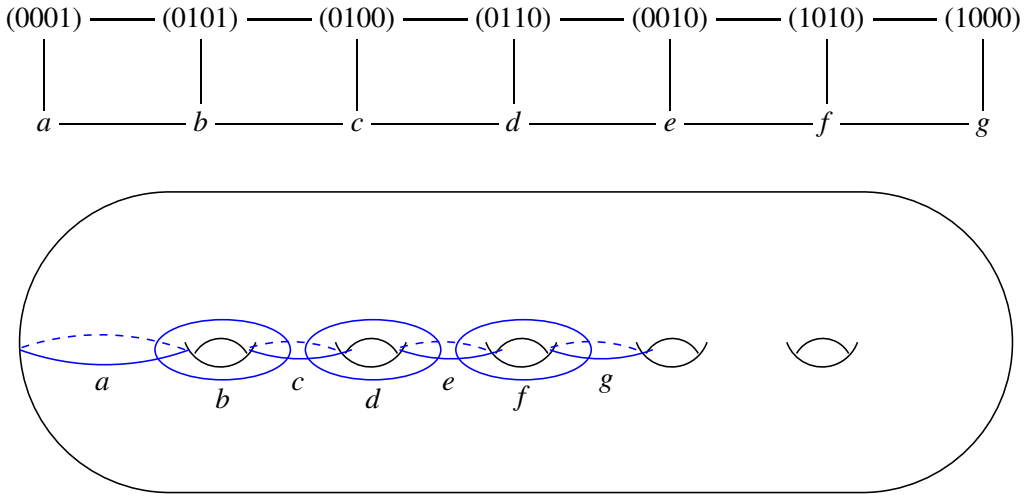
\begin{figure}[ht]
\begin{center} 
\begin{tikzpicture}

\draw[semithick] (-4,0) arc (200:340:0.5);
\draw[semithick] (-2,0) arc (200:340:0.5);
\draw[semithick] (0,0) arc (200:340:0.5);
\draw[semithick] (2,0) arc (200:340:0.5);
\draw[semithick] (4,0) arc (200:340:0.5);

\draw[semithick] (-3.87,-0.15) arc (150:30:0.4);
\draw[semithick] (-1.87,-0.15) arc (150:30:0.4);
\draw[semithick] (0.13,-0.15) arc (150:30:0.4);
\draw[semithick] (2.13,-0.15) arc (150:30:0.4);
\draw[semithick] (4.13,-0.15) arc (150:30:0.4);


\draw[semithick] (5,2) arc (90:-90:2);
\draw[semithick] (-4.2,2) arc (90:270:2);
\draw[semithick] (5,2) -- (-4.2,2);
\draw[semithick] (5,-2) -- (-4.2,-2);


\draw[semithick,blue] (-6.2,-0.1) arc (250:290:3.3);
\draw[semithick,blue,dashed] (-4,-0.1) arc (70:110:3.3);

\draw[semithick,blue] (-3.15,-0.1) arc (250:290:2);
\draw[semithick,blue,dashed] (-2,-0.1) arc (70:110:1.8);

\draw[semithick,blue] (-1.15,-0.1) arc (250:290:2);
\draw[semithick,blue,dashed] (0,-0.1) arc (70:110:1.8);
\draw[semithick,blue] (0.85,-0.1) arc (250:290:2);
\draw[semithick,blue,dashed] (2,-0.1) arc (70:110:1.8);

\draw[semithick, blue] (-3.5,-0.1) ellipse (0.9cm and 0.4cm);
\draw[semithick, blue] (-1.5,-0.1) ellipse (0.9cm and 0.4cm);
\draw[semithick, blue] (0.5,-0.1) ellipse (0.9cm and 0.4cm);

\draw (-5.1,-0.55) node {$a$};
\draw (-3.5,-0.75) node {$b$};
\draw (-2.5,-0.55) node {$c$};
\draw (-1.5,-0.75) node {$d$};
\draw (-.5,-0.55) node {$e$};
\draw (.5,-0.75) node {$f$};
\draw (1.6,-0.55) node {$g$};

\end{tikzpicture}
\end{center}
\caption{An $A_7$-chain of curves\label{Fig:A7_chain}}
\end{figure}

After cutting along the curves labelled $a,c,e,g$ one obtains a torus with 8 boundary components. The remaining curves of the configuration $\Gamma$ can now be drawn on this bordered surface $\Sigma_1^8$.   In particular, if we label by $h$ the generator which comes from the vertex $(1001)$ and which extends the $A_7$-chain to a cycle of 8 curves, as defining the $Br_9^{aff}$-representation of Lemma \ref{lem:affine_braid}, then this looks as in Figure \ref{Fig:A7_chain_and_cut}, where either of the subsurfaces into which $h$ divides $\Sigma_1^8$ could contain the additional handle. 
\begin{figure}[ht]
\begin{center} 
\begin{tikzpicture}[scale=0.7]

\draw[semithick,blue] (-1.2,2) ellipse (0.7cm and 0.2cm);
\draw[semithick,blue] (1.2,2) ellipse (.7cm and 0.2cm);
\draw[semithick,blue, rotate around={-50:(2,1)}] (3,1.8) ellipse (.7cm and 0.2cm);
\draw[semithick,blue, rotate around={50:(-2,1)}] (-3,1.8) ellipse (.7cm and 0.2cm);
\draw[semithick,blue, rotate around={90:(4.3, -1.5)}] (4.3,-1.5) ellipse (.7cm and 0.2cm);
\draw[semithick,blue, rotate around={90:(-4.3, -1.5)}] (-4.3,-1.5) ellipse (.7cm and 0.2cm);
\draw[semithick,blue, rotate around={-50:(-2,-3.8)}] (-3,-4.6) ellipse (.7cm and 0.2cm);
\draw[semithick,blue, rotate around={50:(2,-3.8)}] (3,-4.6) ellipse (.7cm and 0.2cm);

\draw[semithick] (-0.5,-3.5) arc (200:340:0.5);
\draw[semithick] (-0.37,-3.66) arc (150:30:0.4);

\draw[semithick] (-0.5,2) -- (0.5,2);
\draw[semithick] (-2.8,-4.05) arc (230:310:4.33);

\draw[semithick] (1.9,2) arc (200:265:1.1);
\draw[semithick] (-1.9,2) arc (340:275:1.1);
\draw[semithick] (3.7,0.2) -- (4.3,-0.8);
\draw[semithick] (-3.7,0.2) -- (-4.3,-0.8);
\draw[semithick] (4.3,-2.2) -- (3.7,-3);
\draw[semithick] (-4.3,-2.2) -- (-3.7,-3);

\draw[semithick, blue] (1.7,1.88) arc (200:267:1.25);
\draw[semithick,blue] (-1.7,1.88) arc (340:272:1.25);
\draw[semithick,blue] (3.3,0.4) arc (190:230:2.5);
\draw[semithick,blue] (-3.3,0.4) arc (350:310:2.5);
\draw[semithick,blue] (4.15,-1.9) arc (100:193:1.1);
\draw[semithick,blue] (-4.15,-1.9) arc (80:-13:1.1);

\draw[semithick,blue] (1.1,1.8) arc (150:243:4.13); 
\draw[semithick,blue] (-1.1,1.8) arc (390:297:4.13);
\draw[blue] (0.9, -1) node {$h$};
\draw[blue] (-0.9, -1) node {$h$};


\draw (-1.2,2.5) node {$g$}; 
\draw (1.2,2.5) node {$g$};
\draw  (2.15,0.9) node {$f$};
\draw  (-2.15,0.9) node {$f$};
\draw (3.75,0.85) node {$e$};
\draw (-3.75,0.85) node {$e$};
\draw (3.2,-0.45) node {$d$};
\draw (-3.2,-0.45) node {$d$};
\draw (4.8,-1.5) node {$c$};
\draw (-4.8,-1.5) node {$c$};
\draw (3.2, -2.3) node {$b$};
\draw (-3.15, -2.3) node {$b$};
\draw (3.5,-4.0) node {$a$};
\draw (-3.5,-4.0) node {$a$};

\end{tikzpicture}
\end{center}
\caption{The affine braid group generators after $\Sigma$ has been cut open along the $A_7$-chain (the handle may be on the other side of $h$)\label{Fig:A7_chain_and_cut}}
\end{figure}
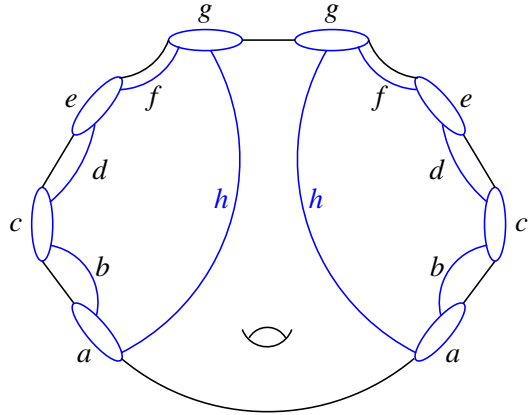

We now consider the two pairs of curves 
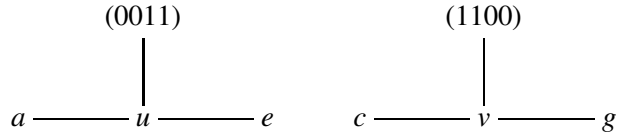
\[
\xymatrix{
& (0011)  \ar@{-}[d] &  & &  (1100)  \ar@{-}[d]  &    \\ 
a \ar@{-}[r] &u \ar@{-}[r]  & e & c \ar@{-}[r] &  v \ar@{-}[r] & g 
}
\]
i.e. $u, v$ correspond to the vertices $(0011)$ respectively $(1100)$, and
\begin{itemize}
\item $u$ meets $\{a,e\}$ but is disjoint from $\{b,c,d,f,g,h,v\}$;
\item $v$ meets $\{c,g\}$ but is disjoint from $\{a,b,d,e,f,h,u\}$.
\end{itemize}
Since $u$ meets $a$, it meets it in precisely one of the two intervals into which $a$ has already been divided by the intersections of $a$ with $b$ and $h$.  These two options are shown in Figure \ref{Fig:affine_generator}. 
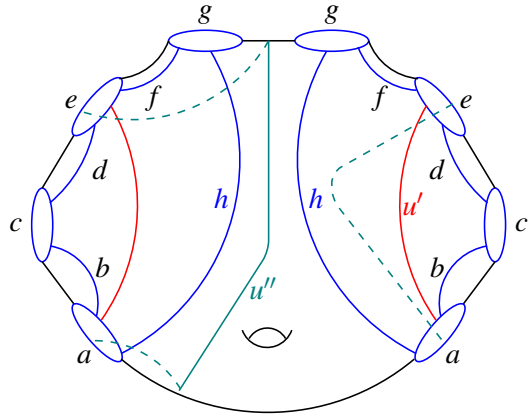
\begin{figure}[ht]
\begin{center} 
\begin{tikzpicture}[scale=0.7]

\draw[semithick,blue] (-1.2,2) ellipse (0.7cm and 0.2cm);
\draw[semithick,blue] (1.2,2) ellipse (.7cm and 0.2cm);
\draw[semithick,blue, rotate around={-50:(2,1)}] (3,1.8) ellipse (.7cm and 0.2cm);
\draw[semithick,blue, rotate around={50:(-2,1)}] (-3,1.8) ellipse (.7cm and 0.2cm);
\draw[semithick,blue, rotate around={90:(4.3, -1.5)}] (4.3,-1.5) ellipse (.7cm and 0.2cm);
\draw[semithick,blue, rotate around={90:(-4.3, -1.5)}] (-4.3,-1.5) ellipse (.7cm and 0.2cm);
\draw[semithick,blue, rotate around={-50:(-2,-3.8)}] (-3,-4.6) ellipse (.7cm and 0.2cm);
\draw[semithick,blue, rotate around={50:(2,-3.8)}] (3,-4.6) ellipse (.7cm and 0.2cm);

\draw[semithick] (-0.5,-3.5) arc (200:340:0.5);
\draw[semithick] (-0.37,-3.66) arc (150:30:0.4);

\draw[semithick] (-0.5,2) -- (0.5,2);
\draw[semithick] (-2.8,-4.05) arc (230:310:4.33);

\draw[semithick] (1.9,2) arc (200:265:1.1);
\draw[semithick] (-1.9,2) arc (340:275:1.1);
\draw[semithick] (3.7,0.2) -- (4.3,-0.8);
\draw[semithick] (-3.7,0.2) -- (-4.3,-0.8);
\draw[semithick] (4.3,-2.2) -- (3.7,-3);
\draw[semithick] (-4.3,-2.2) -- (-3.7,-3);

\draw[semithick, blue] (1.7,1.88) arc (200:267:1.25);
\draw[semithick,blue] (-1.7,1.88) arc (340:272:1.25);
\draw[semithick,blue] (3.3,0.4) arc (190:230:2.5);
\draw[semithick,blue] (-3.3,0.4) arc (350:310:2.5);
\draw[semithick,blue] (4.15,-1.9) arc (100:193:1.1);
\draw[semithick,blue] (-4.15,-1.9) arc (80:-13:1.1);

\draw[semithick,blue] (1.1,1.8) arc (150:243:4.13); 
\draw[semithick,blue] (-1.1,1.8) arc (390:297:4.13);
\draw[blue] (0.9, -1) node {$h$};
\draw[blue] (-0.9, -1) node {$h$};

\draw[semithick,red] (3,0.77) arc (150:215:3.8);
\draw[semithick,red] (-3,0.77) arc (390:325:3.8);
\draw[red] (2.75,-1.1) node {$u'$};

\draw[semithick,teal,dashed] (-3.3,-3.7) arc (90:31:1.9); 
\draw[semithick,teal, rounded corners] (-1.7,-4.65) -- (0,-2) -- (0,2);
\draw[semithick,teal,dashed] (0,2) arc (330:250:3);
\draw[semithick,teal,dashed, rounded corners] (3.5,0.8) -- (1.2,-0.5) -- (1.2, -1) -- (3.3,-3.7);
\draw[teal] (-0.1,-2.7) node {$u''$};

\draw (-1.2,2.5) node {$g$}; 
\draw (1.2,2.5) node {$g$};
\draw  (2.15,0.9) node {$f$};
\draw  (-2.15,0.9) node {$f$};
\draw (3.75,0.85) node {$e$};
\draw (-3.75,0.85) node {$e$};
\draw (3.2,-0.45) node {$d$};
\draw (-3.2,-0.45) node {$d$};
\draw (4.8,-1.5) node {$c$};
\draw (-4.8,-1.5) node {$c$};
\draw (3.2, -2.3) node {$b$};
\draw (-3.15, -2.3) node {$b$};
\draw (3.5,-4.0) node {$a$};
\draw (-3.5,-4.0) node {$a$};

\end{tikzpicture}
\end{center}
\caption{Two possible positions $u'$ and $u''$ for the curve $u$, which meets only $\{a,e\}$\label{Fig:affine_generator}}
\end{figure}
Note that since $u \cap h = \emptyset$, the curve $u$  is confined to live  in one of the two regions into which $h$ divides $\Sigma_1^8$; it might wind non-trivially around the additional handle in the region which contains that handle, but that will have no bearing on the subsequent argument.  The arc $v$ being disjoint from $h$ is constrained similarly, and then $u \cap v = \emptyset$ means that $u$ and $v$ live in different regions of $\Sigma_1^8 \backslash \{h\}$.

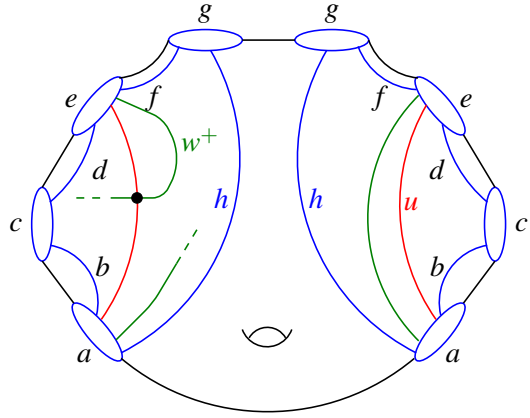
\begin{figure}[ht]
\begin{center} 
\begin{tikzpicture}[scale=0.7]

\draw[semithick,blue] (-1.2,2) ellipse (0.7cm and 0.2cm);
\draw[semithick,blue] (1.2,2) ellipse (.7cm and 0.2cm);
\draw[semithick,blue, rotate around={-50:(2,1)}] (3,1.8) ellipse (.7cm and 0.2cm);
\draw[semithick,blue, rotate around={50:(-2,1)}] (-3,1.8) ellipse (.7cm and 0.2cm);
\draw[semithick,blue, rotate around={90:(4.3, -1.5)}] (4.3,-1.5) ellipse (.7cm and 0.2cm);
\draw[semithick,blue, rotate around={90:(-4.3, -1.5)}] (-4.3,-1.5) ellipse (.7cm and 0.2cm);
\draw[semithick,blue, rotate around={-50:(-2,-3.8)}] (-3,-4.6) ellipse (.7cm and 0.2cm);
\draw[semithick,blue, rotate around={50:(2,-3.8)}] (3,-4.6) ellipse (.7cm and 0.2cm);

\draw[semithick] (-0.5,-3.5) arc (200:340:0.5);
\draw[semithick] (-0.37,-3.66) arc (150:30:0.4);

\draw[semithick] (-0.5,2) -- (0.5,2);
\draw[semithick] (-2.8,-4.05) arc (230:310:4.33);

\draw[semithick] (1.9,2) arc (200:265:1.1);
\draw[semithick] (-1.9,2) arc (340:275:1.1);
\draw[semithick] (3.7,0.2) -- (4.3,-0.8);
\draw[semithick] (-3.7,0.2) -- (-4.3,-0.8);
\draw[semithick] (4.3,-2.2) -- (3.7,-3);
\draw[semithick] (-4.3,-2.2) -- (-3.7,-3);

\draw[semithick, blue] (1.7,1.88) arc (200:267:1.25);
\draw[semithick,blue] (-1.7,1.88) arc (340:272:1.25);
\draw[semithick,blue] (3.3,0.4) arc (190:230:2.5);
\draw[semithick,blue] (-3.3,0.4) arc (350:310:2.5);
\draw[semithick,blue] (4.15,-1.9) arc (100:193:1.1);
\draw[semithick,blue] (-4.15,-1.9) arc (80:-13:1.1);

\draw[semithick,blue] (1.1,1.8) arc (150:243:4.13); 
\draw[semithick,blue] (-1.1,1.8) arc (390:297:4.13);
\draw[blue] (0.9, -1) node {$h$};
\draw[blue] (-0.9, -1) node {$h$};

\draw[semithick,red] (3,0.77) arc (150:215:3.8);
\draw[semithick,red] (-3,0.77) arc (390:325:3.8);
\draw[red] (2.7,-1.1) node {$u$};

\draw[semithick,green!50!black] (2.85,0.95) arc (135:225:3.3);
\draw[semithick, green!50!black, rounded corners] (-3,-1) -- (-2,-1) -- (-1.75,-0.5) -- (-1.75,0) -- (-2,0.5) -- (-2.9,0.9);
\draw[semithick,green!50!black,rounded corners] (-2.9,-3.7) -- (-2.2,-3) -- (-1.75,-2.25);
\draw[semithick,green!50!black,dashed] (-3.2,-1) -- (-3.7,-1);
\draw[semithick,green!50!black,dashed] (-1.75,-2.25) -- (-1.3, -1.5);
\draw[green!50!black] (-1.3,0.1) node {$w^+$};
\draw[fill] (-2.5,-1) circle (0.1);

\draw (-1.2,2.5) node {$g$}; 
\draw (1.2,2.5) node {$g$};
\draw  (2.15,0.9) node {$f$};
\draw  (-2.15,0.9) node {$f$};
\draw (3.75,0.85) node {$e$};
\draw (-3.75,0.85) node {$e$};
\draw (3.2,-0.45) node {$d$};
\draw (-3.2,-0.45) node {$d$};
\draw (4.8,-1.5) node {$c$};
\draw (-4.8,-1.5) node {$c$};
\draw (3.2, -2.3) node {$b$};
\draw (-3.15, -2.3) node {$b$};
\draw (3.5,-4.0) node {$a$};
\draw (-3.5,-4.0) node {$a$};

\end{tikzpicture}
\end{center}
\caption{Impossible curve configurations\label{Fig:first_impossibility}}
\end{figure}

We now try to add the generator $w^+$ corresponding to $(0111)$ in Figure \ref{Fig:first_impossibility}, which hits precisely the curves $\{a,b,c,d,e,u\}$ but not $\{f,g,h,v\}$; we would consider the curve  $w^-$ associated to $(1110)$, which hits $\{c,d,e,f,g,v\}$ but not $\{a,b,h,u\}$, if we had the other choice of $u$ from Figure \ref{Fig:affine_generator} (in which case $v$ lies in the regions occupied by the red copy of $u$ in Figure \ref{Fig:affine_generator}).  Thus, we are now considering a subset of the graph $\Gamma$ with the following intersection pattern, which makes the symmetry between the situations with $\{u,w^+\}$ respectively $\{v,w^-\}$ manifest:
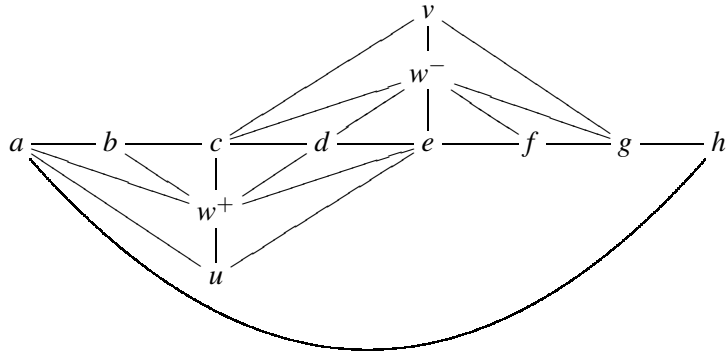
\[
\xymatrix@R=1em{
&&&&v \ar@{-}[d]&&& \\
&&&&w^- &&&\\
a \ar@{-}@/_6.5pc/[rrrrrrr]_{} \ar@{-}[r] \ar@{-}[rrd] \ar@{-}[rrdd] & b \ar@{-}[r] \ar@{-}[rd] & c \ar@{-}[r] \ar@{-}[d] \ar@{-}[rru] \ar@{-}[rruu] & d \ar@{-}[r] \ar@{-}[ru] \ar@{-}[ld] & e \ar@{-}[r] \ar@{-}[u] \ar@{-}[lld] \ar@{-}[lldd] & f \ar@{-}[r] \ar@{-}[lu] & g \ar@{-}[r] \ar@{-}[llu] \ar@{-}[lluu] & h \\
&& w^+ &&&&& \\
&& u \ar@{-}[u] &&&&&
}
\]
The (green) $w^+$-curve meets $u$; the point of intersection is marked by $\bullet$ in Figure \ref{Fig:first_impossibility}. Disjointness from $\{f,g,h\}$ then forces the green curve into either $e$ or $a$ (again, arbitrarily, we have depicted the former case).  It emerges on the other side of $e$, and then it cannot cross $u$ again, since the braid relation means these curves have geometric intersection number 1, nor can it meet $\{f,g,h\}$, so it goes into $a$, but then it enters a region where it can't escape: thus the two dotted ends of $w^+$ can't rejoin.  In the other configuration, where $u$ lies on the `back' of Figure \ref{Fig:A7_chain_and_cut} but $v$ lies on the front, following the green $w^-$-curve  from its unique  intersection with $v$, one runs into the same problem.  In both cases we arrive at a contradiction, i.e. an unrealisable configuration on $\Sigma_5$.  

More formally, label the three regions of Figure \ref{Fig:affine_generator} $A,B,C$ from left to right (note that $A\cup C$ is connected on $\Sigma_1^8$, despite appearances in the figure). Disjointness of $u$ and $v$ ensures that one of these curves lies in $A\cup C$ and the other in $B$.  Suppose $u$ lies in $A \cup C$; it then separates this  into two regions.  The curve $w^+$ must also lie inside $A \cup C$, but cross between its two regions exactly once to have geometric intersection number one with  $u$.  Reversing the roles of $u,v$ respectively $w^+, w^-$ covers the other possible case.

\begin{rmk} There are still four further vertices in $\Gamma$ and we have not appealed to the triangle relation constraints of Remark \ref{rmk:triangle}, so this is far from a `sharp' obstruction.   We have chosen to focus on the commuting / braid relations which define $G(\Gamma)$ since these stay  closer to the usual  realm of Artin and Coxeter groups.
\end{rmk}
 
\begin{cor} \label{cor:disconnected} The monodromy $G(\Gamma) \to Sp(10;\bZ)$ does not factor through the mapping class group of a surface $\Sigma$ of total genus $5$.
\end{cor}

\begin{proof} 
The case in which $\Sigma$ is connected follows from the preceding arguments. If $\Sigma = \sqcup_j \Sigma_j$, then $\Gamma(\Sigma)$ is built out of the mapping class groups $\Gamma_{g_j}$ of the components, together with permutations of diffeomorphic components.   Suppose $\Sigma$ has more than one component of genus $g \in \{1,2\}$. When $k<n$, any homomorphism from $Br_n$ to $\mathrm{Sym}_k$ has cyclic image \cite{Artin}, so all the elements $f_{\eta} \in Br_6$ induce the same permutation of these components. Since the generators act by transvections, this must be the trivial permutation; then since all the generators $f_v$ of $G(\Gamma)$ are conjugate, they necessarily all induce the trivial permutation of $\pi_0\Sigma$.  It follows that, when $\Sigma$ is disconnected, a factorization of the monodromy through $\Gamma(\Sigma)$ is in fact through $\prod_i \Gamma_{g_i}$.  If there is more than one factor, this contradicts Lemma \ref{lem:irred}.
\end{proof}

\begin{rmk} Recall from Remark \ref{rmk:fano} that the monodromy homomorphism does factor through a certain finite index subgroup $\mathrm{Prym}_6$ of $\Gamma_6$. Note that this does \emph{not} say that if factors through $\Gamma_6$ itself, since the homomorphism $\mathrm{Prym}_6 \to Sp(10;\bZ)$ does not extend to $\Gamma_6$. 
\end{rmk}

\bibliographystyle{amsalpha}
\bibliography{mybib}

\end{document}